\documentclass[11 pt]{article}

\usepackage[arrow, matrix, curve]{xy}

\usepackage{amsfonts,amsthm,amssymb,amsmath}

\newtheorem{theorem}{Theorem}

\newtheorem{lemma}[theorem]{Lemma}

\newtheorem{corollary}[theorem]{Corollary}

\newtheorem{problem}{Problem}

\usepackage{color}

\newcommand{\bi}{\mathbf i}
\newcommand{\bj}{\mathbf j}

 \DeclareMathOperator{\card}{card}
 
\DeclareMathOperator{\BH}{BH}

\renewcommand{\thefootnote}{\fnsymbol{footnote}}

\title{{\bf Bohnenblust-Hille
inequalities for Lorentz spaces\\ via interpolation}}

\author{Andreas Defant and Mieczys{\l}aw Masty{\l}o}

\date{}

\begin{document}
\maketitle

\noindent
\renewcommand{\thefootnote}{\fnsymbol{footnote}}
\footnotetext{2010 \emph{Mathematics Subject Classification}:
Primary 46B70, 47A53.} \footnotetext{\emph{Key words and phrases}:
Bohnenblust-Hille inequality, Dirichlet polynomials, Dirichlet
series, Homogeneous polynomials, interpolation spaces, Lorentz
spaces.} \footnotetext{The second named author was supported by
the Foundation for Polish Science (FNP).}

\begin{abstract}
\noindent   We prove that the  Lorentz sequence space
$\ell_{\frac{2m}{m+1},1}$  is, in a~precise sense,
optimal among
all symmetric Banach sequence spaces satisfying
a~Bohnenblust-Hille type inequality for $m$-linear forms or
$m$-homogeneous polynomials on $\mathbb{C}^n$. Motivated by this
result we develop methods for dealing with subtle
Bohnen\-blust-Hille type inequalities in the setting of Lorentz
spaces.~Based on an interpolation approach and the Blei-Fournier
inequalities involving mixed type spaces, we prove multilinear and
polynomial Bohnen\-blust-Hille type inequalities in Lorentz spaces
with subpolynomial and subexponential constants. Improving a
remarkable result of Balasubramanian-Calado-Queff\'{e}lec, we show
an application to the  theory of Dirichlet series.
\end{abstract}

\vspace{5 mm}

\section{Introduction and classical results}

In their seminal article  \cite{BoHi31} Bohnenblust and Hille
proved that there exists a~positive function $f$ on $\mathbb{N}$
such that for every $n$ and every $m$-homogeneous polynomial on
$\mathbb{C}^n$\!, the  $\ell_p$-norm with $p= \frac{2m}{m+1}$ of
the set of its coefficients is bounded above by  the  constant
$f(m)$ times the supremum norm of the polynomial on the unit
polydisc $\mathbb{D}^n$. The initial interest of this result is
that $f(m)$ is independent of the dimension $n$ and, moreover, the
exponent
 $\frac{2m}{m+1}$ is
optimal. This result was a~key point in the celebrated solution by
Bohnenblust and Hille of Bohr's absolute convergence problem for
Dirichlet series (see, e.g., \cite{BoHi31, Bo13_Goett, DeGaMaSe},
or \cite{DeSe14}).

Recently, more sophisticated results were obtained and
successfully applied to verify several long standing conjectures
in the convergence theory  for
Dirichlet series (and intimately related complex analysis in high
dimensions). A~striking improvement was given in
\cite{DeFrOrOuSe11} proving that $f(m)$ in fact grows at most
exponentially in $m$, and a~recent result from \cite{BaPeSe13}
even states that $f(m)$ is subexponential in the sense that  for
every  $\varepsilon>0$ there is a~constant $C(\varepsilon)$ such
that $f(m) \leq C(\varepsilon)(1+\varepsilon)^m$ for each $m\in
\mathbb{N}$. Estimates of this type proved to be useful in many
different areas of analysis (e.g., the modern
$\mathcal{H}_p$-theory of Dirichlet series and (the intimately
connected) infinite dimensional holomorphy (see, e.g.,
\cite{BaDeFrMaSe15} or \cite{DeSe14}), the study of  summing
polynomials in  Banach spaces (see, e.g., \cite{ABDS14,
DeMaSch12}, or \cite{DiSe14}), and even in quantum information
theory (see \cite{Mo12}) and  more generally in Fourier analysis
of Boolean functions. A good general reference in this area is the
recent book of O'Donnell \cite{Donnell}.

Our aim is to prove multilinear and polynomial
Bohnenblust-Hille inequalities in the setting of Lorentz spaces.
In the remaining part of this introduction we give more precise details
on the state of art of BH-inequalities (multilinear and polynomial),
and isolate the two  natural problems we are mainly concerned with.

We will consider Banach sequence spaces
$\big(X(I),\|\cdot\|_X\big)$ of $\mathbb{C}$-valued sequences
$(x_{i})_{i \in I}$ which are defined over arbitrarily given
(index) sets $I$. In what follows Lorentz spaces will play an
important role. Given $1\leq p<\infty$ and $1\leq q\leq \infty$,
the Lorentz space $\ell_{p,q}(I)$ ($\ell_{p,q}$ for short) on
a~nonempty set $I$ consists of all $x=(x_i)_{i\in I}$ for which
the expression
\begin{align} \label{Lorentzsequence}
\|x\|_{\ell_{p,q}} =
\begin{cases}
\Big(\sum_{k\in J} {x_{k}^{*}}^{q} \big(k^{q/p} -
(k-1)^{q/p}\big)^{q}\Big)^{1/q} & \text{if} \quad q<\infty  \\[1ex]
\sup_{k\in J} k^{1/p} x_k^{*} &\text{if }\quad q= \infty
\end{cases}
\end{align}
is finite. Here, as usual, for a~given $x =(x_i)_{i\in I} \in
\ell_{\infty}(I)$, we denote by $x^{*}=(x_j^{*})_{j\in J}$ the
non-increasing rearrangement of $x$ defined by
\[
x_j^{*}= \inf\big\{\lambda>0; \, \card\{i\in I; \,|x_i|>\lambda \}
\leq j\big\}, \quad\, j\in J,
\]
where $J= \{1,...,n\}$ whenever $\card{I}=n$, and $J=\mathbb{N}$
whenever  $I$ is infinite. The expression \eqref{Lorentzsequence}
is a~norm if $q\leq p$, and a quasi norm if $q>p$. In the second
case $\|\cdot\|_{\ell_{p,q}}$ is equivalent to a~norm. Of course,
$\ell_{p,p}$ is the Minkowski space $\ell_{p}$ since the map $x
\mapsto x^{*}$ is an isometry.

The following two finite index sets will be of special interest:
For each $m$, $n \in \mathbb{N}$
\[
\mathcal{M}(m,n) = \big\{ \bi= (i_1, \ldots, i_m) \,; \,i_k \in
\mathbb{N}\,,\, 1 \leq i_k \le n  \big\}
\]
and
\[
\mathcal{J}(m,n) = \big\{ \bj \in \mathcal{M}(m,n) \,;\,  j_1 \leq
j_2\le \ldots \leq j_m \big\}\,.
\]
Below we explain the two inequalities we are interested in, the
so-called multilinear and polynomial Bohnenblust-Hille
inequalities, and we motivate the two problems we intend to
handle.

\bigskip

\noindent{\bf The multilinear BH-inequality.} Given a~Banach
sequence space $X$ (defined over arbitrary index sets) and $m \in
\mathbb{N}$, we denote by
\[
\BH_X^{\text{mult}}(m) \in \big[1, \infty\big]
\]
the best  constant $C \ge 1$ such that  for every $n$ and every complex
matrix $a=( a_\bi)_{\bi \in \mathcal{M}(m,n)}$ we have
\begin{align} \label{BoHiX}
\big\| ( a_\bi)_{\bi \in \mathcal{M}(m,n)} \big\|_X  \leq C   \|a\|_\infty\,,
\end{align}
where
\[
\|a\|_\infty = \sup_{\substack{\|(x^k_i)_{i=1}^n\|_\infty \leq 1\\
1\leq k \leq m}} \,\,\Big| \sum_{\bi =(i_1, \ldots, i_n) \in
\mathcal{M}(m,n)} a_\bi \,\,x^1_{i_1}\ldots x_{i_m}^m \Big|\,.
\]
For the sake of completeness we give a~short review of the history
of the inequalities from \eqref{BoHiX} emphasizing those results,
old and very recent ones, which are of relevance for this article.
(For more on that we once again refer to \cite{DeSe14}.) The case
$m=2$ reflects a famous result of Littlewood \cite{L30}:
\begin{align*}
\label{L}
\BH_{\ell_{\frac{4}{3}}}^{\text{mult}}(2) < \infty.
\end{align*}

\noindent Solving Bohr's  so-called absolute convergence problem
on Dirichlet series Bohnenblust and Hille in \cite{BoHi31} studied
the case of arbitrary $m$ and proved that
\begin{equation}
\label{BH}
\BH_{\ell_{\frac{2m}{m+1}}}^{\text{mult}}(m) < \infty.
\end{equation}

\noindent This result was improved by  Fournier and Blei
\cite{BlFo89, Fo89} showing that even,
\begin{align}
\label{BF} \BH_{\ell_{\frac{2m}{m+1},1}}^{\text{mult}}(m) <
\infty.
\end{align}
In Section 4 we give a~modified version of their proof from
\cite{BlFo89}.

\noindent Finally, it turned out in a~recent article
\cite{BaPeSe13} by Bayart, Pellegrino, and Seoane that
 the constants in \eqref{BH} are
subpolynomial in the following sense: There is a~constant $\kappa
>1$ such that for all $m$ we have
\begin{align} \label{BPS}
\BH_{\ell_{\frac{2m}{m+1}}}^{\text{mult}}(m) \leq
\kappa\,m^{\frac{1-\gamma}{2}},
\end{align}
where $\gamma$ is the Euler-Masceroni constant. Note that there
exits a~uniform constant $C>0$ such that for any finite index sets
$I$
\begin{align}
\label{log}
\big\| \ell_{p}(I)  \hookrightarrow  \ell_{p,1}(I)\big\| \leq C \log (\card{I})\,,
\end{align}
hence by \eqref{BPS} there exits $\delta >1$ such that for each
$m$, $n$ and every matrix $( a_\bi)_{\bi \in \mathcal{M}(m,n)}$,
\begin{align*}
\big\|(a_\bi)_{\bi \in \mathcal{M}(m,n)} \big\|_{\frac{2m}{m+1},1}
\leq m^{\delta}\,(\log n) \|a\|_\infty\,.
\end{align*}
In view of this, and comparing with \eqref{BF} and  \eqref{BPS},
the following natural question appears.

\vspace{2mm}

\begin{problem}
Does there exist a constant $\delta >0$ such that for each $m$ we
have
\[
\BH_{\ell_{\frac{2m}{m+1},1}}^{\text{mult}}(m) \leq m^\delta\,.
\]
\end{problem}
\noindent We provide  far-reaching partial solutions
extending all results mentioned before. The main contributions are
given in the Theorems \ref{main1} and \ref{main2}.

\vspace{2 mm}

\noindent{\bf The polynomial BH-inequality.} Every $m$-homogenous
polynomial
$$P(z) = \sum_{\substack{\alpha \in \mathbb{N}_0^n\\ |\alpha|=m}} c_\alpha z^\alpha$$
in $n$ complex variables $z= (z_1, \ldots, z_n) \in \mathbb{C}^n$
can be uniquely rewritten in the form
\begin{equation} \label{polynomial}
P(z) = \sum_{\bj \in   \mathcal{J}(m,n)} c_\bj \,z_{j_1} \ldots
z_{j_m} \,,
\end{equation}
and we denote its supremum norm by
\[
\|P\|_\infty = \sup_{\substack{\|(z_i)_{i=1}^n\|_\infty \leq 1}}
\,\,\Big| \sum_{\bj =(i_1, \ldots, i_n) \in \mathcal{J}(m,n)}
c_\bj \,\,z_{j_1}\ldots z_{j_m} \Big|\,.
\]
Given a~Banach sequence space $X$ (defined over an arbitrary index
set) and $m \in \mathbb{N}$, we denote by
\[
\BH_X^{\text{pol}}(m) \in \big[1, \infty\big]
\]
the best  constant $C \ge 1$ such that  for every $n$ and every
$m$-homogeneous polynomial $P$ as in \eqref{polynomial} we have
\begin{align} \label{BoHiXX}
\big\| ( c_\bj(P))_{\bj \in \mathcal{J}(m,n)} \big\|_X  \leq C
\|P\|_\infty\,.\end{align} Let us again give a~short review of the
most important  results on such inequalities (for more information
see again \cite{DeSe14}): \noindent Inventing polarization,
Bohnenblust and Hille in \cite{BoHi31} deduced from \eqref{BH}
that
\begin{align}
\label{BHpol} \BH_{\ell_{\frac{2m}{m+1}}}^{\text{pol}}(m) <
\infty.
\end{align}
The fact that $p=\frac{2m}{m+1}$ is optimal here was a~crucial
step in the solution of Bohr's so-called absolute convergence
problem. Again, mainly motivated through problems on the general
theory of Dirichlet series and holomorphic functions in high
dimensions, the first qualitative  improvement of the constants
was done  in \cite{DeFrOrOuSe11}: For every $\varepsilon
>0$ there is a~constant $C(\varepsilon)>0$ such that for all $m$
\begin{align} \label{BHpol1}
\BH_{\ell_{\frac{2m}{m+1}}}^{\text{pol}}(m) \leq C(\varepsilon)
(\sqrt{2}+\varepsilon)^{m}.
\end{align}

\noindent Bayart, Pellegrino, and Seoane proved in \cite{BaPeSe13}
that these constants even are subexponential in the following
sense:
\begin{align} \label{pol3}
\BH_{\ell_{\frac{2m}{m+1}}}^{\text{pol}}(m) \leq C(\varepsilon)
(1+\varepsilon)^{m}.
\end{align}

\noindent We are going to see that a standard polarization
argument extends  \eqref{BHpol} to Lorentz spaces:
\begin{equation} \label{BHpol132}
\BH_{\ell_{\frac{2m}{m+1},1}}^{\text{pol}}(m) < \infty\,,
\end{equation}
but the following problem will turn out to be much more
challenging.
\begin{problem}
To  what extent do
\eqref{BHpol1} and \eqref{pol3} hold  when we replace $\ell_{\frac{2m}{m+1}}$ by the
Lorentz sequence space $\ell_{\frac{2m}{m+1},1}$.
\end{problem}

 Subsequent to the case
of \eqref{BHpol1} our main result is given
in Theorem \ref{polyconstants2}.\\

Why do Lorentz spaces play an essential role within the context of
Bohnenblust-Hille inequalities?  We  prove (see Theorem
\ref{symmetry}) that among all symmetric Banach sequence spaces
$X$ satisfying a~multilinear or polynomial Bohnenblust-Hille
inequality  as in \eqref{BoHiX} or \eqref{BoHiXX} the sequence
space $X= \ell_{\frac{2m}{m+1},1} $ is the smallest one (and in
this sense the ``best").

\section{Preliminaries}
Throughout the paper, for a given finite set $\{X_i\}_{i\in I}$ of
Banach spaces which are all contained in some linear space
$\mathcal{X}$, we denote by $\bigoplus_{i\in I} X_i$ the Banach
space of all $x\in \bigcap_{i\in I} X_i$ equipped with the norm
\[
\|x\|_{\bigoplus_{i\in I} X_i} = \sum_{i\in I}\|x\|_{X_i}.
\]
For each $m\in \mathbb{N}$ we denote by $\mathcal{M}(m)$ and
$\mathcal{J}(m)$ the union of all $\mathcal{M}(m,n)$ and
$\mathcal{J}(m,n)$, $n \in \mathbb{N}$, respectively. We define an
equivalence relation in $\mathcal{M}(m,n)$ in the following way:
$\bi \sim \bj$ if there is a permutation $\sigma$ of $\{1, \ldots,
m\}$ such that $(i_1, \ldots, j_m)=(j_{\sigma(1)}, \ldots,
j_{\sigma(m)})$, and denote by $[\bi]$ the equivalence class of
$\bi\in \mathcal{M}(m,n)$. The following disjoint partition of
$\mathcal{M}(m,n)$ will be very useful:
\[
\mathcal{M}(m,n) = \bigcup_{\bj \in \mathcal{J}(m,n)} [\bj]\,
\]
For $1\leq k\leq m$, let $\mathcal{P}_k(m)$  denote the set of
all subsets of $\big\{1, \ldots, m \big\}$  with cardinality $k$.
We denote the complement of $S \in \mathcal{P}_k(m)$ in
$\{1,\ldots, m\}$ by $\widehat{S}$.
 If $S \in \mathcal{P}_k(m)$, then
$\mathcal{M}(S,n)$ stands for all indices $\bi\colon S \rightarrow
\{1, \ldots,n\}$, and in  the special case $S=\{1, \ldots,k\}$ we
clearly have that $\mathcal{M}(k,n)=\mathcal{M}(S,n)$.  Finally,
for  $\bi \in\mathcal{M}(S,n)$ and $\bj
\in\mathcal{M}(\widehat{S},n)$ we define $\bi\oplus \bj \in
\mathcal{M}(m,n)$ through
\[
\bi\oplus \bj =
\begin{cases}
\bi & \text{ on }  S \\
\bj & \text{ on } \widehat{S}.
\end{cases}
\]
Given  $m,n,k \in \mathbb{N}$ with $1 \leq k < m$ and $1 \leq
p,q \leq \infty$, we define on the space
$\mathbb{C}^{\mathcal{M}(m,n)}$ of all matrices $a=(a_\bi)_{\bi
\in \mathcal{M}(m,n)}$ the norm $\|\cdot\|_{(m,n,k,p,q)}$ by
\[
\|a\|_{(m,n,k,p,q)} = \sum_{S \in \mathcal{P}_k(m)} \Bigg(
\sum_{\bi \in \mathcal{M}(S,n)} \bigg( \sum_{\bj\in
\mathcal{M}(\widehat{S},n)} |a_{\bi \oplus \bj}|^q\bigg)^{p/q}
\Bigg)^{1/p}\,,
\]
and denote the corresponding  Banach space by
\[
\bigoplus_{S \in \mathcal{P}_k(m)}
\ell_p(S)\big[\ell_q(\widehat{S})\big]\,.
\]
Clearly, this is the  $\ell_1$-sum of all Banach spaces
$\ell_p(S)[\ell_q(\widehat{S})]$, where
$\ell_p(S)[\ell_q(\widehat{S})]$ by definition equals
$\mathbb{C}^{\mathcal{M}(m,n)}$ normed by
\[
\|a\|_{\ell_p(S)[\ell_q(\widehat{S})]} =  \Bigg( \sum_{\bi \in
\mathcal{M}(S,n)} \bigg( \sum_{\bj\in \mathcal{M}(\widehat{S},n)}
|a_{\bi \oplus \bj}|^q\bigg)^{p/q} \Bigg)^{1/p}\,.
\]
We will consider (classes) Banach lattices. Of particular
importance are \emph{symmetric} spaces. We recall that a~Banach
lattice $E$ on a measure space $(\Omega, \Sigma, \mu)$ is said to
be symmetric whenever  $g\in E$ and $\|f\|_E =\|g\|_E$ provided
that $\mu_f=\mu_g$ and  $f\in E$. Here $\mu_f$ denotes the
distribution function of $f$ defined by $\mu_f(\lambda) = \mu \{t
\in\Omega; \, |f(t)| > \lambda\}$, $\lambda \ge 0$. Throughout the
paper, by a~Banach sequence lattice on a finite or countable set
$I$ we mean a~real or complex Banach lattice $E$ on the measure
space $(I, 2^{I}, \mu)$ (for short on $I$), where $\mu$ is
counting measure. In the case when $E$ is symmetric, $E$ is said
to be a symmetric Banach $($sequence$)$ space.

A symmetric space $E$ is called \emph{fully symmetric} whenever it
is an exact interpolation space between $L_1(\mu)$ and
$L_{\infty}(\mu)$, i.e., for any linear operator $T\colon L_1(\mu)
+ L_\infty(\mu) \to L_1(\mu) + L_\infty(\mu)$ such that
$\|T\|_{L_1(\mu) \to L_1(\mu)} \leq 1$ and $\|T\|_{L_\infty(\mu)
\to L_\infty}(\mu) \leq 1$, we have that $T$ maps $E$ into $E$ and
$\|T\|_{E \to E}\leq 1$. It is well known that symmetric spaces
that have the Fatou property or have order continuous norm are
fully symmetric (see, e.g., \cite{BS, KPS}).

We will need the concept of discretization of a~Banach lattice.
Let $(\Omega,\Sigma,\mu)$ be a~measure space and let
$d=\{\Omega_k\}_{k=1}^{N} \subset \Sigma$ be a~measurable
partition of $\Omega$, i.e., $\Omega=\bigcup_{k=1}^{N}\Omega_k$
where $\Omega_i\cap \Omega_j=\emptyset$ for each $i,
j\in\{1,\ldots,N\}$ with $i\not=j$. Then, given a~Banach lattice
$X$ on $(\Omega, \Sigma, \mu)$, the discretization $X^d$ is the
Banach space of all simple function $f\in X$ of the form
$f=\sum_{k=1}^{N} \xi_k \chi_{\Omega_k} \in X$ equipped with the
induced norm from $X$.

The notion of Lorentz spaces over arbitrary measure spaces will be
essential in what follows. Given a measure space $(\Omega, \Sigma,
\mu)$ and $0<p <\infty$, $0<q \leq \infty$, the Lorentz space
$L_{p,q}(\Omega, \mu)$ ($L_{p, q}(\Omega)$ or $L_{p,q}$ for short)
is defined to be the space of all (equivalence classes of)
measurable functions $f$ on $\Omega$ equipped with the quasi norm
\[
\|f\|_{L_{p,q}} =
\begin{cases}
\Big(\frac{q}{p} \int_{0}^{\infty} f^{*}(t)^{q} t^{\frac{q}{p} -1}
\,dt\Big)^{1/q} & \text{ if } \quad q<\infty  \\[1ex]
\,\sup_{t>0} t^{1/p} f^{*}(t) &\text{ if }\quad q= \infty\,,
\end{cases}
\]
where $f^{*}$ is the decreasing rearrangement of $f$ defined on
$[0, \infty)$  by
\begin{align*}
\label{rearrangement} f^{*}(t) = \inf \big\{s>0; \, \mu_f(s)\leq
t\big\}
\end{align*}
(We adopt the convention $\inf \emptyset = \infty$.) In the case
when $\Omega=I$ is a non-empty set and $\mu$ is its  counting
measure, the space $L_{p,q}(\Omega, \mu)$  in fact coincides with
the Lorentz sequence space $\ell_{p,q}(I)$ already defined in
\eqref{Lorentzsequence}. Indeed, in this case, given a~function
$f=x$ on $\Omega=I$ we have $x^{*}_k= f^{*}(t)$ for every $t\in
[k-1, k)$, $k\in J$, where $ J= \{1,...,\card{I}\}$ if $I$ is
finite, and $J=\mathbb{N}$ if $I$ is infinite. Thus
$\|f\|_{L_{p,q}} =\|x\|_{\ell_{p,q}}$, where the latter norm is as
defined by the formula \eqref{Lorentzsequence}.

We recall that the K\"othe dual space $(\ell_{p,1})'$ of the
Lorentz space $\ell_{p,1}=\ell_{p,1}(I)$  coincides with the
Marcinkiewicz space $m_{p}$ which consists of all complex
sequences $x=(x_i)_{i\in I}$ such that
\[
\|x\big\|_{m_p} = \sup_{k\in J} \frac{\sum_{j=1}^{k}
x_{j}^{*}}{k^{1/p}} < \infty\,,
\]
and which with this norm forms a Banach space. Moreover, we note
that by standard comparison with the integral of $t^{\alpha}$ on
$[1, N]$, we  for each $N\in \mathbb{N}$ and every $\alpha \in (0,
1)$ have
\begin{align}
\label{estimate}  \sum_{k=1}^{N} \frac{1}{k^{\alpha}} <
\frac{1}{1-\alpha} N^{1-\alpha}\,.
\end{align}
Combining this inequality (for $\alpha = 1/p$) with $x_{k}^{*}
\leq k^{-1/p}\|x\|_{\ell_{p, \infty}},k\in J$ yields
$$m_p = \ell_{p, \infty}$$ up to equivalent norms:
\begin{align*}
\frac{1}{p'}\,\|x\|_{m_p} \leq \|x\|_{\ell_{p,\infty}} \leq
\|x\|_{m_p}, \quad\, x\in \ell_{p, \infty}\,.
\end{align*}
(As usual we write $1/p':= 1- 1/p$.) Many of our arguments will be
based on  interpolation theory. Here we recall some of its basic
concepts and provide some special facts we are going to use.
Recall that if $\vec{A}=(A_0, A_1)$ is a~quasi normed couple, then
for any $a\in A_0 + A_1$ we define the $K$-functional
\[
K(t, a; \vec{A}) = \inf\{\|a_0\|_{A_0} + t\|a_1\|_{A_1}; \,\, a_0
+ a_1 = a\}, \quad\, t>0.
\]
For $0<\theta<1$, $0<q<\infty$, the real interpolation space
$(A_0, A_1)_{\theta, q}$ is the space of all $a\in A_0 +A_1$
equipped with the quasi norm
\[
\|a\|_{\theta, q} = \Big(\int_{0}^{\infty} \big(t^{-\theta} K(t,
a; \vec{A})\big)^{q}\frac{dt}{t}\Big)^{1/q}\,,
\]
with an obvious modification for $q=\infty$.

The following well known and easily verified interpolation
property holds: If $(A_0, A_1)$ and $(B_0, B_1)$ are two
quasi normed couples and $T\colon (A_0, A_1)\to (B_0, B_1)$ and
$T\colon A_0 + A_1 \to B_0 +B_1$ is such that both restrictions
$T\colon A_j \to B_j$ are bounded with the quasi norms $M_j$, then
$T\colon (A_0, A_1)_{\theta, q}\to (B_0, B_1)_{\theta, q}$ is also
bounded, and for its quasi norm $M$ we have
\[
M\leq M_0^{1-\theta} M_{1}^{\theta}.
\]
Lorentz spaces arise naturally in the real interpolation method
since most of their important properties  can be derived from real
interpolation theorems. We briefly review some basic definitions.
The couple $(L_1, L_\infty)$ is especially important for the
understanding of the space $L_{p, q}$. It is well known that for
every $f\in L_1 + L_{\infty}$,
\[
K(t, f; L_1, L_\infty) = \int_{0}^t f^{*}(s)\,ds = tf^{**}(t),
\quad\, t>0.
\]
Hence, for each $\theta \in (0, 1)$
\[
\|f\|_{\theta, q} = \Big(\int_{0}^{\infty} [t^{1-\theta}
f^{**}(t)]^{q}\frac{dt}{t}\Big)^{1/q}\,.
\]
An immediate consequence of Hardy's inequality is the following
well known formula, which states that for $1<p<\infty$, $1\leq
q\leq \infty$ and $\theta = 1- 1/p$
\[
(L_1, L_{\infty})_{\theta, q} = L_{p,q}\,,
\]
and moreover
\[
\frac{1}{p'} \|f\|_{(L_1, L_{\infty})_{\theta, q}} \leq
\|f\|_{L_{p,q}}  \leq  \|f\|_{(L_1, L_{\infty})_{\theta, p}}\,.
\]
Moreover, the following result will be used (which follows from the
more general result stated in  \cite[Theorem 4.3]{Holmstedt}): Let
$1/p = (1-\theta)/p_0 + \theta/p_1$, $0<p_0$, $p_1<\infty$,
$p_0\neq p_1$ and $0<q\leq \infty$. Then, up to equivalent norms,
we have
\[
(L_{p_0}, L_{p_1})_{\theta, q} = L_{p,q}.
\]
More precisely,
\begin{align}
\label{general formula}
\begin{split}
 C^{-1} \theta^{-\text{min}(1/q,
1/p_0)}&(1-\theta)^{-\text{min}(1/q,
1/p_1)}\,\bigg(\frac{p}{q}\bigg)^{1/q} \|f\|_{L_{p,q}}
\\[1ex]&
\,\leq\,
\|f\|_{(L_{p_0}, L_{p_1})_{\theta, q}}
\\[1ex]&
\leq \, C \theta^{-\text{max}(1/q,
1/p_0)}(1-\theta)^{-\text{max}(1/q, 1/p_1)}
\bigg(\frac{p}{q}\bigg)^{1/q}\,\|f\|_{L_{p,q}}\,,
\end{split}
\end{align}
where $C>0$ is a universal constant.

We will also make intensive use of  complex interpolation, and
denote by $[A_0, A_1]_{\theta}$ the complex interpolation spaces
as defined for example in \cite{Ca}. We recall that if $X_0$ and
$X_1$ are two complex Banach lattices on a~measure space $(\Omega,
\Sigma, \mu)$, then
\begin{equation} \label{Calderon}
[X_0, X_1]_{\theta} = X_0^{1-\theta} X_1^{\theta}
\end{equation}
with equality of norms provided one of the spaces has order
continuous norm; here following Calder\'on \cite{Ca} we denote by
$X_0^{1-\theta} X_1^{\theta}$ the Calder\'on space of all $x\in
L^0(\mu)$ such that $|x|\leq \lambda
|x_0|^{1-\theta}|x_1|^{\theta}$ $\mu$-a.e. on $\Omega$ for some
constant $\lambda>0$ and some $x_i\in X_i$ with $\|x_i\|_{X_i}\leq
1$ for $i=0,1$. We put
\[
\|x\|_{X_0^{1-\theta} X_1^{\theta}} = \inf \lambda.
\]

\section{The optimality of Lorentz spaces}

The following theorem motivates our study;  we show that in the
context of multilinear and polynomial Bohnenblust-Hille
inequalities Lorentz spaces are in a~certain sense optimal. Before
we state and prove these results we recall that if $X$ is
a~symmetric Banach sequence space on $I$ and $\chi_A$ denotes the
indicator function of a~set $A\subset I$, clearly $\|\chi_{A}\|_X$
depends only on $\card(A)$. The function $\phi_X(k)=
\|\chi_A\|_X$, where $A\subset I$ with $\text{card}(A)=k$, is
called the \emph{fundamental function} of $X$. It is well known
(see, e.g., \cite[Theorem 2.5.2]{KPS}) that if $1\leq p<\infty$
and $X$ is a~symmetric Banach sequence space on $I$ such that
$\|\chi_{A}\|_X = \card(A)^{1/p}$ for every indicator function
$\chi_A$ (i.e., $\phi_X(k) = k^{1/p}$ for every $A\subset I$ with
$\text{card}(A)=k$), then $\ell_{p, 1} \hookrightarrow X$ with
\[
\|x\|_X \leq \|x\|_{\ell_{p,1}}, \quad\, x\in \ell_{p, 1}\,.
\]
Thus $\ell_{p,1}$ is the smallest symmetric Banach sequence space
on $I$ whose norm coincides with the $\ell_p$-norm on indicator
functions.

\bigskip

\begin{theorem}
\label{symmetry} Fix a~positive integer $m$.~The Lorentz space
$\ell_{\frac{2m}{m+1},1}$ is the smallest symmetric Banach
sequence space $X$ such that $\BH_X^{\text{mult}}(m) <
\infty$.~Also, the Lorentz space $\ell_{\frac{2m}{m+1},1}$ is the
smallest symmetric Banach sequence space $X$ such that
$\BH_X^{\text{pol}}(m) < \infty$.
\end{theorem}

\vspace{1.5 mm}

\begin{proof}
We follow an argument inspired by \cite{BoHi31}. Assume that $X$
is a symmetric Banach sequence space such  that
$\BH_X^{\text{mult}}(m)< \infty$, i.e., for each $n\in \mathbb{N}$
and every complex  matrix $a=( a_\bi)_{\bi \in \mathcal{M}(m,n)}$
we have
\begin{align} \label{BoHiXXX}
\| a \|_X  \leq
\,\BH_X^{\text{mult}}(m)\,   \|a\|_\infty.
\end{align}
It suffices to  show that the fundamental function
\begin{align} \label{fundi}
\phi(n) :=  \Big\|\sum_{i=1}^n e_i\Big\|_{X}, \quad\,n \in
\mathbb{N}\,,
\end{align}
satisfies
\begin{align} \label{main}
\phi(n) \leq C(m)\,n^{^{\frac{m+1}{2m}}}
\end{align}
for each $n\in \mathbb{N}$. For fixed $N$  choose some  $N\times
N$ matrix $(a_{r s})$ such for every $r,s$ we have $\vert a_{rs}
\vert =1$ and $\sum_{k=1}^{N} a_{rk} \overline{a}_{sk} = N
\delta_{rs}$ (e.g. $a_{rs} = e^{2\pi i rs/N },\, 1\leq r,s \leq N
$), and define the matrix $a= (a_\bi)_{\bi\in \mathcal{M}(m,n)}$
by
\begin{align*}
a_{i_{1} \dots i_{m}} = a_{i_{1}i_{2}} \cdots a_{i_{m-1}i_{m}} \,.
\end{align*}
Since $\vert a_{i_{1} \dots i_{m}} \vert=1$, we have $\phi(N^m) =
\| a\|_X$. We now estimate the norm $\|a\|_\infty$. We do first
the trilinear case $m=3$, where the argument becomes more
transparent. We take $x,y,z \in \mathbb{C}^N$ with supremum norm
$\leq 1$, then, using the Cauchy-Schwarz inequality and the
properties of the matrix, we have
\begin{align*}
&
\Big\vert \sum_{i,j,k}   a_{ij} a_{jk} x_{i} y_{j}
  z_{k} \Big\vert
\leq
\sum_{k} \Big\vert \sum_{i,j}   a_{ij} a_{jk} x_{i} y_{j}
\Big\vert \, \vert z_{k} \vert
\\& \leq N^{1/2} \bigg(\sum_{k}
\Big\vert\sum_{i,j}   a_{ij} a_{jk} x_{i} y_{j} \Big\vert^{2} \bigg)^{1/2}
=  N^{1/2} \bigg(\sum_{\substack{i_{1},i_{2} \\ j_{1},j_{2}}}
a_{i_{1}j_{1}} \overline{a}_{i_{2}j_{2}}
x_{i_{1}} \overline{x}_{i_{2}} y_{j_{1}} \overline{y}_{j_{2}} \sum_{k}
a_{j_{1}k}\overline{a}_{j_{2}k} \bigg)^{1/2} \\
& = N^{1/2}  N^{1/2} \bigg(\sum_{\substack{i_{1},i_{2} \\ j}}
a_{i_{1}j} \overline{a}_{i_{2}j}  x_{i_{1}} \overline{x}_{i_{2}}
y_{j} \overline{y}_{j} \bigg)^{1/2}
= N \bigg( \sum_{j}\Big\vert \sum_{i} a_{ij} x_{i} \Big\vert^{2}
\vert y_{j} \vert^{2} \bigg)^{1/2} \\
& \leq N \bigg( \sum_{i_{1}i_{2}} \sum_{j} a_{i_{1}j}
\overline{a}_{i_{2}j} x_{i_{1}} \overline{x}_{i_{2}} \bigg)^{1/2}
= N^{3/2} \bigg( \sum_{i} \vert x_{i} \vert^{2} \bigg)^{1/2} \leq
N^{4/2} \,.
\end{align*}
In the general case we take $z^{(1)} , \dots , z^{(m)} \in
\mathbb{C}^N$, each with supremum norm $\leq 1$, and repeat this
procedure to get
\begin{align}\label{norma ELE}
\bigg\vert  \sum_{i_{1} , \dots , i_{m}=1}^{N} a_{i_{1}i_{2}}
\cdots  a_{i_{m-1}i_{m}} z^{(1)}_{i_{1}} \cdots z^{(m)}_{i_{m}}
\bigg\vert \leq  N^{m/2} \bigg( \sum_{i_{1}} \vert z^{(1)}_{i_{1}}
\vert^{2} \bigg)^{1/2} \leq N^{m/2} N^{1/2} \,.
\end{align}
Hence $\|a\|_\infty  \leq N^{\frac{m+1}{2}}$  for each $N$, and by
\eqref{BoHiXXX} we have $\phi(N^m)\leq BH_X^{\text{mult}}(m)\,
(N^m)^{\frac{m+1}{2m}}$. Since for each positive integer $n$ there
is $N$ such that $N^m \leq n < (N+1)^m$, we finally obtain
\eqref{main}.

 To prove the second statement, we assume that $X$
is a~symmetric Banach sequence space such that for each $n$ and
every
$m$-homogeneous polynomial $P(z)=\sum_{\substack{\alpha \in \mathbb{N}_{0}^{n} \\
\vert \alpha \vert =m}} c_{\alpha} z^{\alpha}$ we have
\begin{align*}
\bigg\| \left(c_\alpha\right)_{\substack{\alpha \in
\mathbb{N}_{0}^{n} \\ \vert \alpha \vert =m}}\bigg\|_X \leq
BH_X^{\text{pol}}(m) \|P\|_\infty\,.
\end{align*}
Following non-trivial ideas of Bohnenblust and Hille from
\cite{BoHi31} it is possible to modify the proof of
the first statement which leads to a sort of
deterministic proof of the second statement.  Here we give an alternative, probabilistic
argument. As in \eqref{fundi} we consider the
fundamental function $\phi(n),  n \in \mathbb{N}$ of $X$. Then by
the Kahane-Salem-Zygmund inequality (see, e.g., Kahane's book
\cite{Ka85}) there is a constant $C_{\text{KSZ}} \geq 1$ such that
for every choice of $N$ there are signs
$\varepsilon_{\alpha}=\pm1$ for which
\[
\sup_{z \in  \mathbb{D}^{N}} \Big\vert  \sum_{\substack{\alpha \in
\mathbb{N}_{0}^{N} \\ \vert \alpha \vert =m}} \varepsilon_{\alpha}
z^{\alpha}  \Big\vert \leq C_{\text{KSZ}} \,\left( N {m+N-1
\choose m}
 \log m \right)^{1/2} \,.
\]
Since the sequence $\big(\phi(N)/N\big)$ is nonincreasing, and for
each $N$ we have
\[
\frac{N^m}{m!} \leq {N + m - 1 \choose m} \leq N^{m}\,,
\]
it follows that $\phi(N^m) \leq m!\,\phi \big({N + m - 1 \choose
m}\big)$ for each $N$. Combining the above
estimates we conclude that for each $N$
\[
\phi(N^m) \leq
BH^{pol}_{X}(m)\,
C_{\text{KSZ}}
\,m!\,
\,\sqrt{\log m}
\,(N^m)^{\frac{m+1}{2m}}.
\]
This easily implies that there exists a constant $C(m)>0$ such
that
\[
\phi(n) \leq C(m) n^{\frac{m+1}{2m}}, \quad\, n\in \mathbb{N}
\]
and the conclusion again follows.
\end{proof}

\section{Multilinear $\pmb{\BH}$-inequalities for Lorentz spaces revisited }

In this section we present a~slightly modified proof of \eqref{BF}
which was first given in the paper by Blei and Fournier
\cite{BlFo89}. We need to prove four preliminary lemmas.

\bigskip

\begin{lemma} \label{lem1}
For each matrix $a=(a_\bi)_{\bi\in \mathcal{M}(m,n)}$ and each
$S\subset \mathcal{M}(m,n)$
\[
\frac{\sum_{\bi \in S}  |a_\bi|}{E(S)} \leq m
\left\|a\right\|_{\ell_{\frac{m}{m-1},\infty}}\,,
\]
where
\[
E(S) := \max_{1 \leq k \leq m } \card\{i_{k}; \, \bi \in S
\}\,.
\]
\end{lemma}

\vspace{1.5 mm}

\begin{proof}
 Clearly
\[
k^{\frac{m-1}{m}} a^{*}_k \leq
\|a\|_{\ell_{\frac{m}{m-1},\infty}}, \quad\, 1 \leq k \leq n^m\,.
\]
Now note that  $\sum_{\bi \in S}  |a_\bi|$ has not more that
$E(S)^m$ summands, and that  $\sum_{k=1}^{E(S)^m} a^*(k)$ sums the
first $E(S)^m$ many largest $|a_\bi|, \bi \in S$. As a~consequence
we obtain by \eqref{estimate} (with $\alpha = 1 -1/m$)
\[
\sum_{\bi \in S}  |a_\bi|  \leq \sum_{k=1}^{E(S)^m}  a^{*}_k \leq
\|a\|_{\ell_{\frac{m}{m-1},\infty}} \sum_{k=1}^{E(S)^m}
k^{-\frac{m-1}{m}} \leq  m
\left\|a\right\|_{\ell_{\frac{m}{m-1},\infty}}E(S)\,,
\]
as desired.
\end{proof}

\bigskip

\begin{lemma} \label{lem2}
For  each matrix
$a=(a_\bi)_{\bi\in \mathcal{M}(m,n)}$ the index set
$\mathcal{M}(m,n)$ splits into a~union of $m$ subsets $S_k$ such
that for every $1 \leq q < \infty$,
\[
\max_{1 \leq k \leq m} \left\| a^{S_k}  \right\|_{ \ell_\infty(\{
k \})\big[\ell_q(\widehat{\{ k \}})\big]} \leq  m^{1/q}
\left\|a\right\|_{\ell_{\frac{qm}{m-1},\infty}}\,,
\]
where for $S \subset \mathcal{M}(m,n)$ we  put  $ a^{S}= a_\bi$ for $\bi \in S$
and $a^{S}= 0$ for $\bi \in S$.
\end{lemma}

\vspace{1.5 mm}

\begin{proof}
Note first that it suffices to show the desired inequality for
$q=1$; for arbitrary $1< q < \infty$ apply the case $q=1$ to
$|a|^{1/q}$ instead of  $a$. In view of Lemma \ref{lem1} we show that there are appropriate sets $S_k$ for which
\[
\max_{1 \leq k \leq m} \left\| a^{S_k}  \right\|_{ \ell_\infty(\{
k \})\big[\ell_1(\widehat{\{ k \}})\big]} \leq  \sup_{S \subset \mathcal{M}(m,n)}
\frac{\sum_{\bi \in S} |a_\bi| }{E(S)}\,,
\]
and  without loss of generality we may
assume that the supremum on the right side  is $\leq 1$. Given $ 1 \leq  k \leq m$,   observe
that
\[
\sum_{\ell =1}^n \,\, \sum_{\substack{\bi \in \mathcal{M}(m,n)\\ i_k = \ell}}  |a_\bi|
\,\leq\, \sum_{\bi \in \mathcal{M}(m,n)}  |a_\bi|
\,\leq\, E(\mathcal{M}(m,n)) =n.
\]
Hence there is some $1 \leq \ell(k) \leq n$ such that for
\[
T_k^1 = \big\{ \bj \in \mathcal{M}(m,n); \, j_k = \ell(k)\big\}
\]
we have
\[
\sum_{\bi \in T_k^1}  |a_\bi| \leq 1.
\]
Then for
\[
N_1  = \mathcal{M}(m,n) \setminus \bigcup_{k=1}^m T^1_k
\]
we obviously  get $E(N_1) \leq n-1$. If we now repeat this
procedure with $N_1$ instead of $\mathcal{M}(m,n)$, then we obtain
$m$ many new index sets  $T_k^2, \, 1 \leq k \leq m$ in $N_1$ for
which
\[
\sum_{\bi \in  T_k^2}  |a_\bi| \leq 1
\]
and
\[
E(N_2) \leq n-2\,\,\, \text{ with  } \,\,\,N_2  =
\bigg(\mathcal{M}(m,n) \setminus \bigcup_{k=1}^m T_k^1\bigg)
\setminus \bigg( \bigcup_{k=1}^m T_k^2\bigg)\,.
\]
Continuing for $j\in \{3, \ldots,n\}$, we find the index sets $T_k^j $,
$1 \leq j \leq n, \,1 \leq k \leq m$ such that
\begin{equation}
\label{Alaba} \sum_{\bi \in T_k^j}  |a_\bi| \leq 1, \quad\, 1 \leq
k \leq m, \,1 \leq j \leq n
\end{equation}
and
\[
E(N_n)=0\,\,\, \text{ with  }
\,\,\,N_n  = \mathcal{M}(m,n)  \setminus \bigcup_{j=1}^n \bigcup_{k=1}^m T_k^j\,.
\]
Define for $1 \leq k \leq m$
\[
S_k  = \bigcup_{j=1}^n  T_k^j\,.
\]
Obviously, we have that $N_n = \emptyset$, and hence
\[
\mathcal{M}(m,n) = \bigcup_{k=1}^m  S_k\,.
\]
Finally, for any $1 \leq k \leq m$
\begin{align*}
\left\| a^{S_k}  \right\|_{ \ell_\infty(\{ k
\})\big[\ell_q(\widehat{\{ k \}})\big]} = \,\sup_{1 \leq j \leq n}
\,\,\,\sum_{\bi \in \mathcal{M}(\widehat{\{k\}},n)} |a^k_{\bi
\oplus j}| \leq
 \sup_{1 \leq j \leq n}  \,\,\sum_{\substack{\bi \in \mathcal{M}(\widehat{\{k\}},n)\\
\bi \oplus j \in\bigcup_{l=1}^n  T_k^l}} |a_{\bi \oplus j}| \leq
1\,.
\end{align*}
Let us comment on the argument for the last estimate: Assume
without loss of generality that $n=2$. Then  by construction,
given $j=1$ or $j=2$,  we have that either $\bi \oplus j \in
T_k^1$ for all $\bi \in \mathcal{M}(\widehat{\{k\}},n)$ or  $\bi
\oplus j \in T_k^2 $  for all $\bi \in
\mathcal{M}(\widehat{\{k\}},n)$. The conclusion follows from
\eqref{Alaba}.
\end{proof}

\bigskip

\begin{lemma} \label{cor2}
For  each matrix
$a=(a_\bi)_{\bi\in \mathcal{M}(m,n)}$ and every $1 \leq q <
\infty$
\[
\|a\|_{\ell_{\frac{qm}{(q-1)m+1},1}} \leq m^{\frac{1}{q}} \sum_{1
\leq k \leq m} \left\| a  \right\|_{ \ell_1(\{ k
\})\big[\ell_{q'}(\widehat{\{ k \}})\big]}.
\]
\end{lemma}

\vspace{1.5 mm}

\begin{proof}  Since  for every $1<r<\infty$ we have $m_{r} = \ell_{r, \infty}$ with
$\|\cdot\|_{\ell_{r, \infty}} \leq \|\cdot\|_{m_r}$ and ($\ell_{r,
1})' = m_r$ isometrically, the required
inequality follows by Lemma \ref{lem2} and a
simple duality argument: Indeed, take a matrix $a$ and sets $S_k$
according to Lemma \ref{lem2}. Then
\begin{align*}
\sum_{\bi \in \mathcal{M}(m,n)} |a_\bi b_\bi|
&
\leq  \sum_{1
\leq k \leq m} \sum_{\bi \in \mathcal{M}(m,n)} |a_\bi b_\bi^{S_k}|
\\
&
\leq  \sum_{1
\leq k \leq m}  \left\| a  \right\|_{ \ell_1(\{ k
\})\big[\ell_{q'}(\widehat{\{ k \}})\big]}\left\| b^{S_k}  \right\|_{ \ell_\infty(\{
k \})\big[\ell_q(\widehat{\{ k \}})\big]}
\\
&
\leq
\max_{1 \leq k \leq m} \left\| b^{S_k}  \right\|_{ \ell_\infty(\{
k \})\big[\ell_q(\widehat{\{ k \}})\big]}
\sum_{1
\leq k \leq m} \left\| a  \right\|_{ \ell_1(\{ k
\})\big[\ell_{q'}(\widehat{\{ k \}})\big]}
\\
&
\leq
m^{1/q}
\left\|b\right\|_{\ell_{\frac{qm}{m-1},\infty}}
\sum_{1
\leq k \leq m} \left\| a  \right\|_{ \ell_1(\{ k
\})\big[\ell_{q'}(\widehat{\{ k \}})\big]}\,,
\end{align*}
the desired conclusion.
\end{proof}

\bigskip

\noindent The last lemma needed is the following so-called mixed
$\BH$-inequality (this is a simple consequence of the multilinear
Khinchine inequality, see e.g., \cite{BaPeSe13, BoHi31}, or
\cite{DeGaMaSe}).

\bigskip

\begin{lemma} \label{lem3}
For  each $n$ and each  matrix $a=(a_\bi)_{\bi \in \mathcal{M}(m,n)}$ we have
\[
\sum_{j=1}^n   \bigg( \sum_{\bi \in
\mathcal{M}(\widehat{\{k\}},n)} |a_{\bi \oplus j}|^2\bigg)^{1/2}
\leq \sqrt{2}^{m-1} \|a\|_\infty, \quad\, 1 \leq k \leq m.
\]
\end{lemma}

\noindent Combining Lemmas \ref{cor2} ($q=2$) and \ref{lem3} gives
the proof of \eqref{BF}. As a by-product we get the following
estimate for the  constant
\[
\BH_{\ell_{\frac{2m}{m+1},1}}^{\text{mult}}(m) \leq m^{1/2} \sqrt{2}^{m-1}\,.
\]
We note a disadvantage of this proof, it does not give  polynomial growth
of $\BH_{\ell_{\frac{2m}{m+1},1}}^{\text{mult}}(m)$ in $m$ as we have for
$\BH_{\ell_{\frac{2m}{m+1}}}^{\text{mult}}(m)$ in \eqref{BPS}.

\bigskip

\subsection{Polynomial growth -- part I}
We are going to give a first  improvement of the result from  \eqref{BPS}. Our estimate shows that the symmetric Banach
sequence space $$X=\ell_{\frac{2m}{m+1},\frac{2(m-1)}{m}}$$
satisfies the BH-inequality from (\ref{BoHiX}) with a constant growing subpolynomially
in $m$. It is important to note that $X$ is strictly larger than
the Lorentz space $\ell_{\frac{2m}{m+1},1}$,  however, $X$ has the same fundamental function
as $\ell_{\frac{2m}{m+1},1}$ which of course
fits with Theorem \ref{symmetry}.

\bigskip

\begin{theorem} \label{main1}
 There exists a~constant $\delta >0$ such that for
each $m$,
\[
 \BH_{\ell_{\frac{2m}{m+1},\frac{2(m-1)}{m}}}^{\text{mult}}(m) \leq m^\delta\,.
\]
\end{theorem}
\noindent  The proof combines ideas and tools from \cite{BlFo89,
BoHi31, L30} with some more recent ones from \cite{BaPeSe13}. The
following lemma, the proof of which is explicitly included in the
proof of \cite[Proposition 3.1]{BaPeSe13}, is crucial. For $1\le p
\le 2$ we write  $A_p\ge 1$ for the  best constant in the
Khinchine-Steinhaus inequality: For each choice of finitely many
$\alpha_1, \ldots, \alpha_N \in \mathbb{C}$
\[
\|(\alpha_k)_{k=1}^N\|_{\ell_2}  \leq A_p \Bigg( \int_{\mathbb{T}^N}
\Big| \sum_{k=1}^N \alpha_k z_k \Big|^p dz  \Bigg)^{1/p}\,,
\]
where $dz$ stands for the normalized Lebesgue measure on the
$N$-dimensional torus $\mathbb{T}^N$. Recall that $A_p \leq
\sqrt{2}$ for all $1\le p \le 2$.

\begin{lemma} \label{lemma1}
For each $n$, each matrix $a=( a_\bi)_{\bi \in \mathcal{M}(m,n)}$, and
 each  $1 \le k < m$ we have
\begin{align*}
\|a\|_{(m,n,k,\frac{2k}{k+1},2)} \leq
A_{\frac{2k}{k+1}}^{m-k} \,\,
\BH_{\ell_{\frac{2k}{k+1}}}^{\text{mult}}(k)\,\,\|a \|_\infty\,.
\end{align*}
\end{lemma}

\noindent The second lemma needed is an immediate consequence of \cite[Theorem 7.2]{BlFo89}.

\medskip

\begin{lemma} \label{lemma2}
For each $1 \leq q < \infty$  there is a constant $C_q\ge 1$ such
that for each $1 \leq t < q$ and  each matrix $a=(a_\bi)_{\bi\in
\mathcal{M}(m,n)}$
\[
\|a \|_{\ell_{\frac{mqt}{mq+t-q},t}} \,\leq \,C_q\, m\,
\|a \|_{(m,n,m-1,t,q)}\,.
\]
\end{lemma}

\vspace{1.5 mm}

\begin{proof}[Proof of Theorem~\ref{main1}]
For $q=2$ and $t=\frac{2(m-1)}{m}$ we have $ \frac{mqt}{mq+t-q} =
\frac{2m}{m+1}. $ Hence, given a matrix $a=( a_\bi)_{\bi \in
\mathcal{M}(m,n)}$, Lemma \ref{lemma2} yields
\[
\|a\|_{\ell_{ \frac{2m}{m+1},\frac{2(m-1)}{m}}} \leq C_2 m
\|a\|_{(m,n,m-1,\frac{2(m-1)}{m},2)}.
\]
Moreover, by Lemma \ref{lemma1} we have
\begin{align*}
\|a\|_{(m,n,m-1,\frac{2(m-1)}{m},2)}
\leq
A_{\frac{2(m-1}{m}}\,\BH_{\ell_{\frac{2(m-1)}{m}}}^{\text{mult}}(m-1)\,\|a \|_\infty\,.
\end{align*}
Combining with \eqref{BPS} we conclude (because $A_p \leq
\sqrt{2}$ for each $1 \leq p \leq 2$) that
\[
\|a\|_{\ell_{ \frac{2m}{m+1},\frac{2(m-1)}{m}}} \leq C_2 m
\sqrt{2}\kappa  (m-1)^{\frac{1-\gamma}{2}} \|a \|_\infty \,,
\]
 as required.
\end{proof}

\subsection{Polynomial growth -- part II}

In this section we use complex and real interpolation as well as
results from Fournier's article \cite{Fo89} to improve
Theorem~\ref{main1} considerably (Theorem \ref{main2}). The
starting point of what we intend to prove is the following result.

\begin{lemma} \label{Fournier}
For each $m,n,k \in \mathbb{N}$ with $1 \leq k \leq m$  we have
that
\[
\Big\| \bigoplus_{S \in \mathcal{P}_k(m)}
\ell_1(S)\big[\ell_\infty(\widehat{S})\big] \,\, \hookrightarrow \,\,
\ell_{\frac{m}{k},1}(\mathcal{M}(m,n)) \Big\| \leq {m \choose
k}^{-1}\,.
\]
\end{lemma}

\begin{proof}
A variant of this result  is mentioned without proof in \cite[p.~69]{Fo89} (the
special case $k=1$ is given in \cite[Theorem 4.1]{Fo89}; for the
general case analyze the proof of \cite[Theorem 4.1]{Fo89} and use
in particular \cite[Theorem 3.3]{Fo89} instead of \cite[Theorem
3.1]{Fo89})  in combination with Cauchy's
inequality.
\end{proof}

\vspace{2 mm}

We will need the following obvious technical result; since we here
are interested in precise norm estimates, we prefer  to  include
a~proof.

\begin{lemma}
\label{intersection} Let $J$ be a~finite set, and let $Y$ and
$X_{j}$, $j\in J$ be Banach lattices on a~measure space $(\Omega,
\Sigma, \mu)$. Then \,\,$\bigoplus_{j\in J} \big( X_{j}^{1-\theta}
Y^{\theta}\big) = \Big(\bigoplus_{j\in J} X_{j}\Big)^{1-
\theta}Y^{\theta}$ for every $\theta \in (0, 1)$ with
\begin{align*}
\Big\|\bigoplus_{j\in J} \big(X_{j}^{1- \theta}Y^{\theta}\big)
\hookrightarrow \Big(\bigoplus_{j\in J} X_{j}\Big)^{1-\theta}
Y^{\theta}\Big\| \leq {\rm{\card{J}}}
\end{align*}
and
\begin{align*}
\Big\|\Big(\bigoplus_{j\in J} X_{j}\Big)^{1-\theta} Y^{\theta}
\hookrightarrow \bigoplus_{j\in J} \big(X_{j}^{1-
\theta}Y^{\theta}\big)\Big\| \leq {\rm{\card{J}}}.
\end{align*}
\end{lemma}

\begin{proof} Let $x\in \bigoplus_{j\in J} \big(X_{j}^{1-
\theta}Y^{\theta}\big)$ with norm less than $1$. Since
$\|x\|_{X_{j}^{1- \theta}Y^{\theta}} < 1$ for each $j \in J$,
there exist $y_j \in Y$, $x_j \in X_j$ with $\|y_j\|_{Y}\leq 1$,
$\|x_j\|_{X_j} \leq 1$ for each $j\in J$ such that
\[
|x| \leq |x_j|^{1-\theta}|y_j|^{\theta}, \quad\, j\in J.
\]
This implies
\begin{align*}
|x|  \leq \big(\min_{k\in J} |x_k|\big)^{1-\theta} \big(\max_{k\in
J} |y_k|\big)^{\theta}.
\end{align*}
Clearly, $\big\|\min_{k\in J} |x_k|\big\|_{\bigoplus_{j\in J} X_j}
\leq \sum_{j\in J} \|x_j\|_{X_j} \leq {\rm{\card{J}}}$ and
$\|\max_{k\in J} |y_k|\|_Y \leq {\rm{\card{J}}}$ yield
\[
x\in \Big(\bigoplus_{j\in J} X_j\Big)^{1-\theta} Y^{\theta}
\]
with
\[
\|x\|_{(\bigoplus_{j\in J} X_j)^{1-\theta} Y^{\theta}} \leq
{\rm{\card{J}}}.
\]
This shows the first estimate from our statement. The proof of the
second statement is straightforward.
\end{proof}
Now we use real and complex interpolation to deduce, from Lemma
\ref{Fournier}, the following result.

\begin{lemma} \label{real-complex}
For each $m,n,k \in \mathbb{N}$ with $1 \leq k \leq m$  we have
\[
\Big\| \bigoplus_{S \in \mathcal{P}_k(m)}
\ell_{\frac{2k}{k+1}}(S)\big[\ell_2(\widehat{S})\big] \,\, \hookrightarrow
\,\, \ell_{\frac{2m}{m+1},\frac{2k}{k+1}}(\mathcal{M}(m,n)) \Big\|
\leq 2\,{m \choose k}^{3/2}.
\]
\end{lemma}

\vspace{1.5 mm}

\begin{proof}
We claim that the following norm estimate holds:
\begin{align}
\label{uno} \Big\|\bigoplus_{S \in \mathcal{P}_k(m)}
\ell_{1}(S)[\ell_2(\widehat{S})] \,\, \hookrightarrow
\,\,\ell_{\frac{2m}{m+k},1}(\mathcal{M}) \Big\| \leq  \sqrt{{m
\choose k}},
\end{align}
where $\mathcal{M} = \mathcal{M}(m,n)$. Indeed, combining complex
interpolation first with Lemma \ref{intersection} \big(with norm $
{m \choose k}$\big) and then with Lemma \ref{Fournier} \big(with
norm $ {m \choose k}^{-1/2}$\big)
 we obtain
\begin{align*}
\bigoplus_{S \in \mathcal{P}_k(m)}
\ell_{1}(S)[\ell_2(\widehat{S})] & = \bigoplus_{S \in
\mathcal{P}_k(m)} \ell_1(S)\left[\left[\ell_1(\widehat{S}),
\ell_\infty(\widehat{S})\right]_{\frac{1}{2}}\right] =
\bigoplus_{S \in \mathcal{P}_k(m)}
\left[\ell_1(S)[\ell_1(\widehat{S})],
\ell_1(S)[\ell_\infty(\widehat{S})]\right]_{\frac{1}{2}}
\\[1ex]
& = \bigoplus_{S \in \mathcal{P}_k(m)} \left[\ell_1(\mathcal{M}),
\ell_1(S)[\ell_\infty(\widehat{S})]\right]_{\frac{1}{2}}
\stackrel{ \leq   {m \choose k}}{\quad \hookrightarrow \quad }
 \Big[\ell_1(\mathcal{M}), \bigoplus_{S \in
\mathcal{P}_k(m)}\ell_1(S)[\ell_\infty(\widehat{S})]\Big]_{\frac{1}{2}}
\\[2ex]
&
\stackrel{ \leq   {m \choose k}^{-1/2}}{\quad \hookrightarrow \quad }
\big[\ell_1(\mathcal{M}),
\ell_{\frac{m}{k},1}(\mathcal{M})\big]_{\frac{1}{2}} =
\ell_{\frac{2m}{m+k},1}(\mathcal{M}).
\end{align*}
Observe that here the last formula
holds with equality of norms; to see this note that for every $1<
p < \infty$ and $0<\theta < 1$ we have by \eqref{Calderon}
\[
E:= \left[ \ell_1(\mathcal{M}), \ell_{p,1}(\mathcal{M}) \right]_{\theta}
= \ell_1(\mathcal{M})^{1-\theta}\ell_{p,1}(\mathcal{M})^{\theta}\,.
\]
Taking K\"othe duals we obtain $E'=
\ell_\infty(\mathcal{M})^{1-\theta} (m_{p}(\mathcal{M}))^{\theta}
= (m_{p})^{\frac{1}{\theta}}$ which for $\theta = \frac{1}{2}$ and
$p = \frac{m}{k}$ gives $E'= m_{\frac{2m}{m-k}}(\mathcal{M})\,, $
and by duality
\[
E= \ell_{\frac{2m}{m+k},1}(\mathcal{M})\,.
\]
This proves the claim from \eqref{uno}. Now for $\theta_k = \frac{k-1}{k}$ we have
\[
\left[ \ell_1(S), \ell_2(S) \right]_{\theta_k} = \ell_{\frac{2k}{k+1}}(S).
\]
Hence we deduce from  \eqref{uno} and again Lemma
\ref{intersection} that
\begin{align*}
\bigoplus_{S \in \mathcal{P}_k(m)}
\ell_{\frac{2k}{k+1}}(S)[\ell_2(\widehat{S})] & = \bigoplus_{S \in
\mathcal{P}_k(m)} \left[ \ell_1(S), \ell_2(S)
\right]_{\theta_k}[\ell_2(\widehat{S})] = \bigoplus_{S \in
\mathcal{P}_k(m)} \left[ \ell_1(S)[\ell_2(\widehat{S})],
\ell_2(S)[\ell_2(\widehat{S})] \right]_{\theta_k}\\[1ex]
& = \bigoplus_{S \in \mathcal{P}_k(m)} \left[
\ell_1(S)[\ell_2(\widehat{S})], \ell_2(\mathcal{M})
\right]_{\theta_k} \stackrel{ \leq   {m \choose k}}{\quad
\hookrightarrow \quad } \left[  \bigoplus_{S \in
\mathcal{P}_k(m)}\ell_{1}(S)[\ell_2(\widehat{S})],\,
\ell_2(\mathcal{M})\right]_{\theta_k} \\[1ex]
& \stackrel{ \leq   {m \choose k}^{\frac{1-\theta_k}{2}}}{\quad
\hookrightarrow \quad } \big[\ell_{\frac{2m}{m+k},1}(\mathcal{M}),
\ell_2(\mathcal{M})\big]_{\theta_k}\,,
\end{align*}
and so the norm of the inclusion map  is less or equal than
\[
{m \choose k} {m \choose k}^{\frac{1-\theta_k}{2}} = {m \choose
k}^{1+ \frac{1}{2k}} \leq {m \choose k}^{3/2}.
\]
We now need the following equality:
\[
\big[\ell_{\frac{2m}{m+k},1}(\mathcal{M}),
\ell_2(\mathcal{M})\big]_{\theta_k} =  \ell_{\frac{2m}{m+1},
\frac{2k}{k+1}}
\]
with
\[
\Big\|\big[\ell_{\frac{2m}{m+k},1}(\mathcal{M}),
\ell_2(\mathcal{M})\big]_{\theta_k} \hookrightarrow
\ell_{\frac{2m}{m+1}, \frac{2k}{k+1}}(\mathcal{M}) \Big\| \leq 2.
\]
In fact, from \eqref{Calderon} it follows that for $1\leq q_j \leq
p_j<\infty$ with $j=0,1$ and for $\theta\in (0,1)$ we have
\[
\big [\ell_{p_0, q_0}, \ell_{p_1, q_1}\big]_{\theta} = (\ell_{p_0,
q_0})^{1-\theta} (\ell_{p_1, q_1})^{\theta}.
\]
And further on  for  $1/p = (1-\theta)/p_0 + \theta/p_1$ and  $1/q =
(1-\theta)/q_0 + \theta/q_1$  it can be shown similarly as in the non-atomic case in \cite[Lemma
4.1]{GM} that in the atomic case  we have
\[
(\ell_{p_0, q_0})^{1-\theta} (\ell_{p_1, q_1})^{\theta} =
\ell_{p,q}
\]
with
\[
\big \|\big (\ell_{p_0, q_0})^{1-\theta}(\ell_{p_1, q_1})^{\theta}
\hookrightarrow \ell_{p, q}\big\| \leq 2^{1/p}.
\]
Thus taking $\theta=\frac{k-1}{k}$, $q_0=1$, $p_0= \frac{2m}{m+k}$
and $p_1=q_1=2$, we obtain the required embedding. Combining all
together, we finally arrive at
\[
\Big\|\bigoplus_{S \in \mathcal{P}_k(m)}
\ell_{\frac{2k}{k+1}}(S)[\ell_2(\widehat{S})] \hookrightarrow
\ell_{\frac{2m}{m+1}, \frac{2k}{k+1}} \Big\| \leq 2\,{m \choose
k}^{3/2}\,,
\]
which completes the proof.
\end{proof}

\noindent A combination of  \eqref{BPS}, Lemma \ref{lemma1} and Lemma \ref{real-complex}
leads to the following substantial improvement of Theorem
\ref{main1}.

\begin{theorem} \label{main2}
For each $m,k \in \mathbb{N}$ with $1 \leq k \leq m$  we have
\[
\BH^{\text{mult}}_{\ell_{\frac{2m}{m+1},\frac{2k}{k+1}}}(m) \leq 2
{m \choose k}^{3/2} \,\,A_{\frac{2k}{k+1}}^{m-k}
\,\,\BH^{\text{mult}}_{\ell_{\frac{2k}{k+1}}}(k)\,.
\]
In particular, for each $k$ there is some $\delta(k)>0$ such that
for every  $m > k $
\[
\BH^{\text{mult}}_{\ell_{\frac{2m}{m+1},\frac{2(m-k)}{m-k+1}}}(m)
\leq \, m^{\delta(k)}.
\]
\end{theorem}

\section{The polynomial $\pmb{\BH}$-inequality for Lorentz spaces}
Let us start with a standard polarization argument showing how the
multilinear BH-inequality in Lorentz spaces from \eqref{BF}
transfers to a polynomial BH-inequality in Lorentz spaces (as
already stated in \eqref{BHpol132}).

\bigskip

\begin{theorem}\label{polycase}
Given $m \in \mathbb{N}$, there is a constant $C>0$ such that for
every $m$-homogeneous polynomial $P = \sum_{\bj \in
\mathcal{J}(m,n)} c_\bj z_{j_1} \ldots z_{j_m}$ in $n$ complex
variables we have
\[
\big\| ( c_\bj)_{\bj \in \mathcal{J}(m,n)} \big\|_{\ell_{\frac{2m}{m+1},1}}  \leq  C \|P\|_\infty\,\,;
\]
in other terms,
\[
 \BH_{\ell_{\frac{2m}{m+1},1}}^{\text{pol}}(m) < \infty\,.
\]
\end{theorem}

\vspace{1.5 mm}

\begin{proof}
Take some $m$-homogeneous polynomial $P$  as above, and let  $a=
(a_\bi)_{\bi \in \mathcal{M}(m,n)}$ be the associated symmetric
matrix. Then for every $\bj \in \mathcal{J}(m,n)$ we have
\[
c_{\bj}=\text{card}[\bj] \, a_{\bj}\,,
\]
and by standard polarization
\[
\|a\|_\infty \leq \frac{m^m}{m!}\|P\|_\infty\,.
\]
Obviously,
\begin{align*} \label{log2}
\big\| \ell_{p,1}(\mathcal{M}(m,n))  \hookrightarrow  \ell_{p,1}(\mathcal{J}(m,n)),
(b_\bi)_{\bi \in \mathcal{M}(m,n)} \mapsto (b_\bj)_{\bj \in \mathcal{J}(m,n)}\big\| \leq 1\,.
\end{align*}
Combining all this we obtain
\begin{align*}
\big\| \big(c_{\bj}\big)_{\bj  \in \mathcal{J}(m,n)}   \big\|_{\frac{2m}{m+1},1}
&
=
\big\| \big( \text{card}[\bj]a_{\bj}\big)_{\bj  \in \mathcal{J}(m,n)}   \big\|_{\ell_{\frac{2m}{m+1},1}}
\\&
\leq
\big\| \big( \text{card}[\bi]a_{\bi}\big)_{\bi  \in \mathcal{M}(m,n)}   \big\|_{\ell_{\frac{2m}{m+1},1}}
\\&
\leq m! \big\| \big( a_{\bi}\big)_{\bi  \in \mathcal{M}(m,n)}   \big\|_{\ell_{\frac{2m}{m+1},1}}
\\&
\leq m! \BH_{\ell_{\frac{2m}{m+1},1}}^{\text{mult}}(m)\|a\|_\infty
\leq m^m \BH_{\ell_{\frac{2m}{m+1},1}}^{\text{mult}}(m)  \|P\|_\infty\,,
\end{align*}
which is the estimate  we aimed at.
\end{proof}

\bigskip

\subsection{Hypercontractive growth}

We now improve the preceding theorem by showing that for $X
=\ell_{\frac{2m}{m+1},1}$ the constant $\BH_{X}^{\text{pol}}(m)$
in fact has hypercontractive  growth in $m$;  this extends
\eqref{BHpol1} from Minkowski spaces $\ell_{\frac{2m}{m+1}}$ to
Lorentz spaces $\ell_{\frac{2m}{m+1},1}$.

\bigskip

\begin{theorem}\label{polyconstants2}
For every $\varepsilon >0$ there is a constant $C(\varepsilon)>0$
such that for each $m$
\[
\BH_{\ell_{\frac{2m}{m+1},1}}^{\text{pol}}(m) \leq
C(\varepsilon)\, \big(\sqrt{2}  + \varepsilon\big)^m\,.
\]
\end{theorem}

\noindent Our proof needs four preliminary lemmas. The
understanding of the following diagonal operator
\begin{align*}
\label{log22}
D(m,n)\colon \mathbb{C}^{\mathcal{M}(m,n),s} \,\, \hookrightarrow
\,\,\mathbb{C}^{\mathcal{J}(m,n)}\,,\,\,\, (a_\bi)_{\bi \in
\mathcal{M}(m,n)} \mapsto (\text{card} [\bj]
^{\frac{m+1}{2m}}a_\bj)_{\bj \in \mathcal{J}(m,n)}
\end{align*}
will turn out to be crucial; here
$\mathbb{C}^{\mathcal{M}(m,n),s}$ stands for all symmetric
matrices in $\mathbb{C}^{\mathcal{M}(m,n)}$, i.e., all matrices
$(a_\bi)_{\bi \in \mathcal{M}(m,n)}$ for which $a_\bi = a_\bj$
whenever $\bj \in [\bi] $. Moreover, denote for $1 < p < \infty$
by $\ell^s_{p,1}(\mathcal{M}(m,n))$  the subspace
$\mathbb{C}^{\mathcal{M}(m,n),s}$ of
$\ell_{p,1}(\mathcal{M}(m,n))$, and define similarly for $1 \leq p
< \infty$ the subspace $\ell^s_{p}(\mathcal{M}(m,n))$.

\vspace{2 mm}

In Lemma \ref{diagonal} we will use interpolation in order to
establish norm estimates for these diagonal operators in Lorentz
sequence spaces. In order to do so, we need another technical
lemma on real interpolation.

\begin{lemma}
\label{kfunctional} Let $X_0$, $X_1$ be fully symmetric spaces on
a~measure space $(\Omega, \Sigma, \mu)$. If $X_{0}^{d}$ and
$X_1^{d}$ are discretizations of $X_0$ and $X_1$ generated by the
same measurable partition of $\Omega$, then for every $\theta \in
(0, 1)$ and $1\leq q \leq \infty$ the inclusion map ${\rm{id}}
\colon (X_0^{d}, X_1^{d})_{\theta, q} \to (X_0, X_1)_{\theta, q}$
is an isometric isomorphism, i.e.,
\[
\|f\|_{(X_0^{d}, X_1^{d})_{\theta, q}} = \|f\|_{(X_0,
X_1)_{\theta, q}}, \quad\, f\in (X_0^{d}, X_1^{d})_{\theta, q}.
\]
\end{lemma}

\vspace{1.5mm}

\begin{proof}
Let $\{\Omega_k\}_{k=1}^{N} \subset \Sigma$ be a given measurable
partition of $\Omega$. Define the linear map
\[
P \colon L_1(\mu) +L_{\infty}(\mu) \to L_1(\mu) +
L_{\infty}(\mu)\, ,\,\,\, f \mapsto
 \sum_{k=1}^{N} \Big(\frac{1}{\mu(\Omega_k)}\,\int_{\Omega_k}
f\,d\mu\Big)\,\chi_{\Omega_k}.
\]
Since $P\colon (L_1(\mu), L_{\infty}(\mu)) \to  (L_1(\mu),
L_{\infty}(\mu))$ with $\|P\|_{L_1(\mu) \to L_1(\mu)} \leq 1$ and
$\|P\|_{L_{\infty}(\mu) \to L_{\infty}(\mu)} \leq 1$,  and $X_0$
and $X_1$ are fully symmetric, it follows that
\[
P\colon (X_0, X_1)\to (X_0^{d}, X_1^{d})
\]
with $\|P\|_{X_j \to X_j^{d}} \leq 1$ for $j\in \{0,1\}$. This
implies that for every $f\in X_0^{d} + X_1^{d}$ we have (by $P(f)
= f$)
\[
K(t, f; X_0^{d}, X_1^{d}) = K(t, Pf; X_0, X_1) \leq K(t, f; X_0,
X_1), \quad\, t>0.
\]
Since the opposite inequality is obvious, the required statement
follows.
\end{proof}

\vspace{2 mm}

The following result will be essential in what follows.

\begin{lemma}
\label{diagonal} There is a uniform constant $L>0$ such that for
each $m$ and $n$
\[
\Big\|D(m,n)\colon
\ell^s_{\frac{2m}{m+1},1}\left(\mathcal{M}(m,n)\right) \,\,
\hookrightarrow
\,\,\ell_{\frac{2m}{m+1},1}\left(\mathcal{J}(m,n)\right)\Big\|
\leq L \, m\,.
\]
\end{lemma}

\begin{proof} The proof is based on interpolation, and the short
forms $\mathcal{M}=\mathcal{M}(m,n)$ as well as
$\mathcal{J}=\mathcal{J}(m,n)$ will be used. We claim that
\begin{align}
\label{Destimate} \big\|D(m,n)\colon \ell_1^s(\mathcal{M}) \to
\ell_1(\mathcal{J})\big\|\leq 1, \quad\, \big \|D(m,n)\colon
\ell_2^s(\mathcal{M}) \rightarrow \ell_2(\mathcal{J})\big\| \leq \sqrt{m}.
\end{align}
Indeed, for every $a \in
\mathbb{C}^{\mathcal{M}(m,n),s}$ we have
\begin{align*}
\big\|D(m,n) a \big\|_{\ell_1(\mathcal{J})} & = \sum_{\bj \in
\mathcal{J}} \text{card}[\bj]^{\frac{m+1}{2m}}|a_\bj|
%\\&
= \sum_{\bj \in \mathcal{J}} \text{card}[\bj]^{\frac{m+1}{2m}-1}
\text{card}[\bj]|a_\bj|
\\&
\leq  \sum_{\bj \in \mathcal{J}}   \text{card}[\bj]|a_\bj| =
\sum_{\bi \in \mathcal{M}}  |a_\bi| =\left\| a
\right\|_{\ell_1^s(\mathcal{M})} \,,
\end{align*}
and
\begin{align*}
\big\|D(m,n) a \big\|_{\ell_2(\mathcal{J})} & = \Big( \sum_{\bj
\in \mathcal{J}}   \text{card}[\bj]^{\frac{m+1}{m}} |a_\bj|^2
\Big)^{1/2}
%\\&
= \Big( \sum_{\bj \in \mathcal{J}}
\text{card}[\bj]^{\frac{m+1}{m}-1} \text{card}[\bj] |a_\bj|^2
\Big)^{1/2}
\\&
= (m!)^{\frac{1}{2m}} \Big( \sum_{\bj \in \mathcal{J}}
\text{card}[\bj] |a_\bj|^2 \Big)^{1/2} \leq \sqrt{m}
\Big(\sum_{\bi \in \mathcal{M}}  |a_\bi|^2 \Big)^{1/2} =
\sqrt{m}\,\left\|a \right\|_{\ell^s_2(\mathcal{M})} \,
\end{align*}
which  proves \eqref{Destimate}. We now apply the two sided norm
estimate from \eqref{general formula}. In the special case when
$p_0=q_0=1$, $p_1=q_1=2$, $q=1$, $\theta= \frac{m-1}{m}$, we have
$p=\frac{2m}{m+1}$  and in particular $1\leq
(p/q)^{1/q}=\frac{2m}{m+1}<2$. Then for $I =\mathcal{M}(m,n)$ or
$I =\mathcal{J}(m,n)$,
\[
\left(\ell_1(I), \ell_2(I)\right)_{\frac{m-1}{m}, 1} =
\ell_{\frac{2m}{m+1},1}(I)\,,
 \]
and there is $C>0$ such that for all $a \in
\mathbb{C}^{\mathcal{M}(m,n),s}$,
\begin{align} \label{point1}
\frac{m^{\frac{3}{2}}}{C(m-1)}\,
\|a\|_{\ell_{\frac{2m}{m+1},1}(I)} \leq \|a\|_{\left( \ell_1(I) ,
\ell_2(I)\right)_{\frac{m-1}{m}, 1}} \leq \frac{C m^2}{m-1}
\|a\|_{\ell_{\frac{2m}{m+1},1}(I)} \,.
\end{align}
It follows by Lemma \ref{kfunctional} that
\begin{align}
\label{point2} \|a\|_{\left( \ell_1^s(\mathcal{M}),
\ell_2^s(\mathcal{M})\right)_{\frac{m-1}{m}, 1}} = \|a\|_{\left(
\ell_1(\mathcal{M}) , \ell_2(\mathcal{M})\right)_{\frac{m-1}{m},
1}}, \quad\ a \in \mathbb{C}^{\mathcal{M}(m,n),s}.
\end{align}
Now we  interpolate; we recall that for every
operator $T$ between interpolation couples $(A_0,A_1)$ and
$(B_0,B_1)$, and every $0 < \theta < 1$ we have
\[
\big \|T\colon (A_0,A_1)_{\theta,1} \rightarrow
(B_0,B_1)_{\theta,1}\big\| \leq \big\|T\colon A_0 \rightarrow
B_0\big\|^{1-\theta} \big\|T\colon A_1 \rightarrow
B_1\big\|^{\theta}\,.
\]
In particular,
\begin{align*}
& \big\|D(m,n)\colon \left( \ell_1^s(\mathcal{M}),
\ell_2^s(\mathcal{M}) \right)_{\frac{m-1}{m},1} \rightarrow
\left(\ell_1(\mathcal{J}), \ell_2(\mathcal{J})
\right)_{\frac{m-1}{m},1} \big\|
\\&
\leq  \big\|D(m,n)\colon \ell_1^s(\mathcal{M})  \rightarrow
\ell_1(\mathcal{J})\big\|^{\frac{1}{m}} \big\|D(m,n)\colon
\ell_2^s(\mathcal{M})  \rightarrow
\ell_2(\mathcal{J})\big\|^{\frac{m-1}{m}} \,.
\end{align*}
As a consequence we obtain that for every $a \in
\mathbb{C}^{\mathcal{M}(m,n),s}$,
\begin{align*} \label{point3}
& \frac{m^{\frac{3}{2}}}{C (m-1)}\,\big\| D(m,n)
a\big\|_{\ell_{\frac{2m}{m+1},1}(\mathcal{M})}
%\\&
\stackrel{\eqref{point1}}{\leq} \big\|D(m,n) a\big\|_{\left(
\ell_1(\mathcal{J}),\ell_2(\mathcal{J})\right)_{\frac{m-1}{m},1} }
\\&
\leq \left\| D(m,n)\colon \ell_1^s(\mathcal{M})  \rightarrow
\ell_1(\mathcal{J})\right\|^{\frac{1}{m}} \left\| D(m,n)\colon
\ell_2^s(\mathcal{M})  \rightarrow
\ell_2(\mathcal{J})\right\|^{\frac{m-1}{m}} \|a\|_{\left(
\ell_1^s(\mathcal{M}),\ell_2^s(\mathcal{M})\right)_{\frac{m-1}{m},1}
}
\\&
\stackrel{\eqref{point2}}{=} \left\|D(m,n)\colon
\ell_1^s(\mathcal{M})  \rightarrow
\ell_1(\mathcal{J})\right\|^{\frac{1}{m}} \left\| D(m,n)\colon
\ell_2^s(\mathcal{M}) \rightarrow
\ell_2(\mathcal{J})\right\|^{\frac{m-1}{m}} \|a\|_{\left(
\ell_1(\mathcal{M}),\ell_2(\mathcal{M})\right)_{\frac{m-1}{m},1} }
\\&
\stackrel{\eqref{point1}}{\leq} \left\| D(m,n)\colon
\ell_1^s(\mathcal{M})  \rightarrow
\ell_1(\mathcal{J})\right\|^{\frac{1}{m}}
 \left\|D(m,n)\colon \ell_2^s(\mathcal{M})  \rightarrow  \ell_2(\mathcal{J})\right\|^{\frac{m-1}{m}}
\frac{C m^2}{m-1} \| a\|_{\ell_{\frac{2m}{m+1},1}(\mathcal{J})}\,.
\end{align*}
Combining the above estimates with \eqref{Destimate}, we conclude
that for every $a \in \mathbb{C}^{\mathcal{M}(m,n),s}$,
\[
\big\|D(m,n) a\|_{\ell_{\frac{2m}{m+1},1}(\mathcal{M})} \leq C^2
\,\sqrt{m} \sqrt{m}^{\frac{m-1}{m}}  \|
a\|_{\ell_{\frac{2m}{m+1},1}(\mathcal{J})} \leq C^2 \,m \|
a\|_{\ell_{\frac{2m}{m+1},1}(\mathcal{J})},
\]
and this completes the proof.
\end{proof}

\noindent For $1 \leq p \leq 2$ define $S_p>0$ to be the best
constant $C>0$ in the Khinchine-Steinhaus inequality for
$m$-homogeneous polynomials (see, e.g., \cite{Ba02} or also
\cite{DeGaMaSe}): For every $m$-homogeneous polynomial $P$ on
$\mathbb{C}^n$ we have
\[
\Bigg( \int_{\mathbb{T}^n} |P(z)|^2 d z\Bigg)^{1/2} \leq S_p^m
\Bigg( \int_{\mathbb{T}^n} |P(z)|^p d z\Bigg)^{1/p}\,;
\]
we will here only  use the fact that  $S_1 \leq \sqrt{2}$. In what
follows we will need the following lemma (see \cite[Lemma
6.6]{DeGaMaSe}) (implicitly contained in \cite{BaPeSe13}), however
only in the case $k=1$.

\bigskip

\begin{lemma} \label{THE lemma}
Let $P = \sum_{\bj \in \mathcal{J}(m,n)} c_\bj z_{j_1} \ldots
z_{j_m}$ be a~$m$-homogeneous polynomial in $n$ variables, and let
$a=(a_{\mathbf{i}})_{\mathbf{i} \in \mathcal{M}(m,n)}$  be its
associated symmetric  matrix. Then for every  $S \in
\mathcal{P}_k(m)\,,\,1 \leq k \leq m $ we have
\begin{align*}
\Bigg( \sum_{_{\bi \in \mathcal{M}(S,n)}} \bigg( \sum_{\bj \in
\mathcal{M}(\widehat{S},n)} {\rm{card}}\,[\bj]\, \big| a_{\bi \oplus
\bj} \big| ^2\bigg)^{\frac{1}{2} \,\, \frac{2k}{k+1}}
\Bigg)^{\frac{k+1}{2k} } \leq   S_{\frac{2k}{k+1}}^{m-k} \,\,
\frac{(m-k)!  m^m}{(m-k)^{m-k} m!}\,\,
B_{\ell_{\frac{2k}{k+1}}}^{\text{mult}}(k) \,\, \|P\|_\infty\,.
\end{align*}
\end{lemma}

\noindent The fourth  lemma is an immediate consequence of
\cite[Theorem 3.3]{BlFo89};  here we will use only
the case $q=2$.

\bigskip

\begin{lemma} \label{lemmaX}
Given $1 \leq q < \infty$,  there is a constant $C_q\ge 1$
such that for every matrix $a=(a_\bi)_{\bi\in \mathcal{M}(m,n)}$
\[
\|a\|_{\ell_{\frac{mq}{m+q-1},1}} \,\leq\, C_q \,m\,
\|a\|_{(m,n,1,1,q)}\,.
\]
\end{lemma}

\bigskip

\noindent  We are now ready to give the proof Theorem
\ref{polyconstants2}.

\vspace{1.5 mm}

\begin{proof}[Proof of Theorem \ref{polyconstants2}] Assume that $P$ is  an $m$-homogeneous
polynomial  on $\mathbb{C}^n$ with coefficients $( c_\bj)_{\bj
\in \mathcal{J}(m,n)}$, and denote the coefficients of the
associated symmetric $m$-linear form $A$ by $( a_\bi)_{\bi \in
\mathcal{M}(m,n)}$. We have the simple fact that for all $\bi \in
\mathcal{M}(\{1\},n)$ and all $\bj\in
\mathcal{M}(\widehat{\{1\}},n)$
\[
\text{card} [\bi \oplus \bj] \leq m \,\text{card} [ \bj]\,.
\]
Hence we deduce  from Lemma \ref{diagonal}, Lemma  \ref{lemmaX}
($q=2$) and Lemma \ref{THE lemma} ($k=1$) that for each $m$ and
$n$
\begin{align*}
\big\| \big( c_\bj\big)_{\bj \in \mathcal{J}(m,n)}
\big\|_{\frac{2m}{m+1},1} & = \big\| \big( \text{card}
[\bi]a_\bi\big)_{\bi \in \mathcal{J}(m,n)}
\big\|_{\frac{2m}{m+1},1}
\\[1ex]
& \le L \, m \big\| \big( \text{card}
[\bi]^{1-\frac{m+1}{2m}}a_\bi\big)_{\bi \in \mathcal{M}(m,n)}
\big\|_{\frac{2m}{m+1},1} \\[1ex]
& \le L \, m C_2 m \Big\|\Big( \text{card}
[\bi]^{1-\frac{m+1}{2m}}a_\bi \Big)_{\bi \in \mathcal{M}(m,n)}
\Big\|_{(m,n,1,1,2)} \\[1ex]
& = L \, m C_2 m \max_{S \in \mathcal{P}_1(m)} \sum_{\bi \in
\mathcal{M}(\{1\},n)} \Bigg( \sum_{\bj\in
\mathcal{M}(\widehat{\{1\}},n)} \big| \text{card} [\bi \oplus
\bj]^{\frac{m-1}{2m}}a_{\bi \oplus \bj}\big|^2\Bigg)^{1/2} \\[1ex]
& \leq L \, m C_2 m \max_{S \in \mathcal{P}_1(m)} \sum_{\bi \in
\mathcal{M}(\{1\},n)} \Bigg( \sum_{\bj\in
\mathcal{M}(\widehat{\{1\}},n)} \big|\big(m\text{card}
[\bj]\big)^{\frac{m-1}{2m}}a_{\bi \oplus \bj}\big|^2\Bigg)^{1/2}
\\[1ex]&
\le L \, m C_2 m m^{\frac{m-1}{2m}} \max_{S \in \mathcal{P}_1(m)}
\sum_{\bi \in \mathcal{M}(\{1\},n)} \Bigg( \sum_{\bj\in
\mathcal{M}(\widehat{\{1\}},n)} \text{card}
[\bj]^{\frac{m-1}{m}}\big| a_{ \bj}\big|^2\Bigg)^{1/2} \\[1ex]
%\end{align*}
%\begin{align*}
& \le L \, m C_2 m m^{\frac{m-1}{2m}} \max_{S \in
\mathcal{P}_1(m)} \sum_{\bi \in \mathcal{M}(\{1\},n)} \Bigg(
\sum_{\bj\in \mathcal{M}(\widehat{\{1\}},n)} \text{card}
[\bj]\big|a_{\bj}\big|^2\Bigg)^{1/2} \\[1ex]
& \le L \, m C_2 m m^{\frac{m-1}{2m}} \sqrt{2}^{m-1} \times
\frac{(m-1)!  m^m}{(m-1)^{m-1} m!} \times
B_{\ell_{1}}^{\text{mult}}(1) \times \|P\|_\infty\,.
\end{align*}
This completes the argument.
\end{proof}

\bigskip

\subsection{The Balasubramanian-Calado-Queff\'{e}lec result revisited}

In this section we improve a remarkable result by
Balasubramanian-Calado-Queff\'effelec \cite{BaCaQu06}. By
$\mathcal{P}(^mc_0)$ we denote the linear space of all
$m$-homogeneous continuous polynomials on $c_0$ which together
with the supremum norm on the open unit ball in $c_0$ forms
a~Banach space. On the subspace $c_{00}$ of all finite sequences
in $c_0$ each  such polynomial has a unique monomial series
decomposition $P(z)=\sum_{|\alpha|=m} c_\alpha(P)
z^{\alpha}\,,\,\, z \in c_{00}$ (or, in different notation, $P(z)
= \sum_{\bj \in  \mathcal{J}(m)} c_\bj \,z_\bj\,,\,\, z \in
c_{00}$). A~Dirichlet series $D=\sum_n a_n n^{-s}$ is said to be
$m$-homogeneous whenever $[a_n \neq 0 \Rightarrow n=
\mathfrak{p}^{\alpha}]$ ($\mathfrak{p}$ the sequence of primes).
All $m$-homogeneous Dirichlet series $D=\sum_n a_n n^{-s}$ which
converge on $[\text{Re}>0]$ and are such that the holomorphic
function $D(s)= \sum_{n=1}^\infty a_n \frac{1}{n^{s} }\,, \, s \in
[\text{Re}>0]$ is bounded, form (together with the supremum norm
on $[\text{Re}>0]$) the Banach space $\mathcal{H}^m_\infty$.

It is remarkable that there is a~unique isometric isomorphism
\[
\mathfrak{B}\colon \mathcal{P}(^mc_0) \rightarrow
\mathcal{H}^m_\infty, \quad\, P = \sum_{|\alpha|=m} c_\alpha(P)
z^{\alpha} \mapsto D=\sum_n a_n n^{-s}
\]
such that $c_\alpha = a_n$ whenever  $n=\mathfrak{p}^{\alpha}$.
(For more information see \cite{DeGaMaSe, DeSe14}, or
\cite{QQ13}.) Then the  following theorem is  an immediate
consequence of this identification and Theorem \ref{polycase}.

\bigskip

\begin{theorem}\label{BCQrevisited}
For every Dirichlet series $\sum_n a_n \frac{1}{n^s} \in
\mathcal{H}^m_\infty$ we have $\big( a_n^*\big) \in
\ell_{\frac{2m}{m-1},1}$\,.
\end{theorem}

\noindent Note that for every  sequence $a=(a_n)\in
\ell_{\frac{2m}{m+1},1}$ we have
\[
\sum_{n=1}^\infty |a_n| \frac{1}{n^{\frac{m-1}{2m}}}  \leq
\sum_{n=1}^\infty a_n^* \frac{1}{n^{\frac{m-1}{2m}}} \asymp
\|a\|_{\ell_{\frac{2m}{m+1}, 1}} < \infty\,.
\]
In \cite{BaCaQu06} it is proved that for every Dirichlet series
$\sum_{n=1}^{\infty} a_n \frac{1}{n^s} \in \mathcal{H}^m_\infty$
\begin{align}
\label{BCQ} \sum_{n=1}^\infty |a_n| \frac{(\log
n)^\frac{m-1}{2}}{n^{\frac{m-1}{2m}}} < \infty\,.
\end{align}
In addition it is shown that the exponent in the log-term is
optimal. A natural question appears: How is this result related
with the estimate from Theorem \ref{BCQrevisited}? To see this let
$\ell_1(\omega)$ be the weighted $\ell_1$-space with the weight
$\omega = (\omega_n)$ given by
\begin{align} \label{weight}
\omega_n = \frac{(\log n)^\frac{m-1}{2}}{n^{\frac{m-1}{2m}}},
\quad\, n \in \mathbb{N}\,.
\end{align}
We observe that  $\ell_1(\omega)$ is different from
$\ell_{\frac{2m}{m+1},1}$\,: In fact, if we would have
$\ell_1(\omega) \subset \ell_{\frac{2m}{m+1},1}$, or equivalently
$\ell_1 \subset \ell_{\frac{2m}{m+1},1}(\omega^{-1})$, then by the
closed graph theorem
\[
\sup_{n\in \mathbb{N}} \|e_n\|_{\ell_{\frac{2m}{m+1},1}(\omega^{-1})} <
\infty\,.
\]
But since for each $n\in \mathbb{N}$
\begin{align*}
\left\|e_n\right\|_{\ell_{\frac{2m}{m+1},1}(\omega^{-1})} =
\left\|\frac{e_n}{\omega_n}\right\|_{\ell_{\frac{2m}{m+1},1}} =
\frac{n^{\frac{m-1}{2m}}}{(\log n)^{\frac{m-1}{m}}},
\end{align*}
we get a~contradiction. Similarly, if $\ell_{\frac{2m}{m+1},1}
\subset \ell_1(\omega)$, then there  would exist a~constant $C>0$
such that for each  $N \in \mathbb{N}$,
\[
\sum_{n=1}^N \frac{(\log n)^\frac{m-1}{2}}{n^{\frac{m-1}{2m}}} =
\Big\|\sum_{n=1}^N e_n \Big\|_{\ell_1(\omega)}  \leq C \Big\|
\sum_{n=1}^N e_n \Big\|_{\ell_{\frac{2m}{m+1},1} } = C\,
N^{\frac{m-1}{2m}}\,,
\]
which is again impossible. We conclude the paper with the
following formal improvement of Theorem \ref{BCQrevisited} and the
Balasubramanian-Calado-Queff\'{e}lec result \eqref{BCQ}:

\begin{corollary}
For each $m\in \mathbb{N}$ and every Dirichlet series
$\sum_{n=1}^{\infty}\,a_n \frac{1}{n^s} \in \mathcal{H}^m_\infty$,
\begin{align*} \label{BCQ+}
(a_n)_n \in \ell_1(\omega) \cap \ell_{\frac{2m}{m+1},1}\,\,,
\end{align*}
where the weight $\omega$ is given by \eqref{weight}.
\end{corollary}

\vspace{2.5 mm}

\noindent
Institut f\"ur Mathematik \\
Carl von Ossietzky Universit\"at \\
Postfach 2503 \\
D-26111 Oldenburg, Germany

\vspace{0.5 mm}

\noindent E-mail: andreas.defant@uni-oldenburg.de

\vspace{3 mm}

\noindent Faculty of Mathematics and Computer Science\\
Adam Mickiewicz University; and Institute of Mathematics\\
Polish Academy of Science (Pozna\'n branch)\\
Umultowska 87, 61-614 Pozna{\'n}, Poland

\vspace{0.5 mm}
\noindent E-mail: mastylo$@$math.amu.edu.pl
\end{document}